\documentclass{amsart}
\let\<=\langle
\let\>=\rangle
\let\\=\backslash
\let\sse=\subseteq
\let\noi=\noindent

\let\limply=\Longrightarrow

\newsymbol\varnothing 203F
\let\void=\varnothing

\font\srm=cmr6 
\def\cd{{\hbox{\srm{\#}}}}

\def\0{\{0\}}

\def\span{{{\rm span}}}
\def\smallfrac#1#2{{\textstyle{\frac{#1}{#2}}}}
\def\conv{{\;\longrightarrow\;}}
\def\sconv{{{\kern-1pt\buildrel_{\scriptstyle s}\over\conv}}}
\def\uconv{{{\kern-1pt\buildrel_{\scriptstyle u}\over\conv}}}

\def\B{{\mathcal B}}
\def\H{{\mathcal H}}
\def\M{{\mathcal M}}
\def\N{{\mathcal N}}
\def\R{{\mathcal R}}
\def\BH{{\B[\H]}}

\def\CC{{\mathbb C\kern.5pt}}
	\def\DD{{\mathbb D}}
\def\TT{{\mathbb T}}

\newtheorem{theorem}{Theorem}[section]
\newtheorem{corollary}{Corollary}[section]
\newtheorem{proposition}{Proposition}[section]
\theoremstyle{definition}

\begin{document}

\vglue-50pt\noi
\hfill{\it Functional Analysis, Approximation and Computation}\/,
{\bf 10}--3 (2018) 21--26

\vglue20pt
\title{Residual Spectrum of Power Bounded Operators}
\author{A. Mello}
\address{\vskip-12pt Rural Federal University of Rio de Janeiro,
         Nova Igua\c cu, RJ, Brazil}
\email{aleksandrodemello@yahoo.com.br}
\author{C.S. Kubrusly}
\address{\vskip-12pt Applied Mathematics Department, Federal University,
         Rio de Janeiro, RJ, Brazil}
\email{carloskubrusly@gmail.com}
\renewcommand{\keywordsname}{Keywords}
\date{July 2, 2018}

\begin{abstract}
The residual spectrum of a power bounded operator lies in the open unit
disk.
\end{abstract}

\maketitle

%%%%%%%%%%%%%%%%%%%%%%%%%%%%%%%%%%%%%%%%%%%%%%%%%%%%%%%%%%%%%% SECTION 1
\vskip-10pt\noi
\section{Introduction}
A well-known open question on Hilbert-space operators asks whether a power
bounded operator has the residual spectrum included in the unit circle$.$
The purpose of this paper is to offer an answer to this so far open
question$.$ The answer follows from known results and is obtained by
elementary arguments.

\vskip6pt
It has been known for a long time that the above question has an
affirmative answer if power bounded is restricted to contractions$.$ All
proofs for the contraction case are elementary$.$ We first give a new but
still elementary proof for the contraction case, which shows that attempts
to extend the contraction case towards the power bounded case might lead
to a false start$.$ The power bounded case requires a different (although
still elementary and based on a well-known result) start.

\vskip6pt
A possible new start that leads to a proof for the power bounded case is
the Ergodic Theorem for power bounded operators$.$ After this, there are
certainly a few different (but similar) paths to proceed$.$ We have
chosen what we consider to be an elementary path in the proof of
Theorem 4.1$.$ Applications are explored in Sections 5 and 6.

%%%%%%%%%%%%%%%%%%%%%%%%%%%%%%%%%%%%%%%%%%%%%%%%%%%%%%%%%%%%%% SECTION 2
\section{Notation}

Throughout the paper $\H$ will stand for an infinite-dimensional, complex
(not necessarily separable) Hilbert space$.$ The inner product in $\H$
will be denoted by ${\<\cdot;\cdot\>}.$ By an operator on $\H$ we mean a
bounded linear transformation of $\H$ into itself$.$ Let $\BH$ stand for
the $C^*$-algebra of all operators on $\H.$ Both norms in $\H$ or in $\BH$
will be denoted by the same symbol ${\|\cdot\|}.$ An operator ${T\in\BH}$
is an isometry if ${\|Tx\|=\|x\|}$ for every ${x\in\H}$, it is unitary if
it is an invertible isometry, a contraction if ${\|Tx\|\le\|x\|}$ for every
${x\in\H}$ (i.e., ${\|T\|\le1}$), and power bounded if
$$
\sup_n\|T^n\|<\infty.
$$
Equivalently (by the Banach--Steinhaus Theorem), if
${\sup_n\|T^nx\|<\infty}$ for every ${x\in\H}$ (otherwise it is called
power unbounded)$.$ It is clear that every isometry is a contraction, and
every contraction is power bounded$.$ A completely nonunitary contraction
is an operator on $\H$ for which its restriction to any reducing subspace
is not unitary$.$ For any operator ${T\in\BH}$, let $\N(T)=T^{-1}(\0)$
denote its kernel and let $I$ stand for the identity in $\BH.$ The set
$\sigma_{\kern-1ptP}(T)
=\big\{{\lambda\in\CC}\!:{\N(\lambda I-T)\ne\0}\big\}$
is the point spectrum of $T$ (i.e., the set of all eigenvalues of $T).$
Let ${T^*\!\in\BH}$ stand for the adjoint of ${T\in\BH}.$ The residual
spectrum of $T$ is the set \cite[p.5]{MDOT}
$$
\sigma_{\kern-1ptR}(T)=\sigma_{\kern-1ptP}(T^*)^*\\\sigma_{\kern-1ptP}(T).
$$
Here $\Lambda^{\!*}=\{{\overline\lambda\in\CC}\!:{\lambda\in\Lambda}\}$ is
the set of all complex conjugates of points in a set ${\Lambda\sse\CC}.$
Let $\DD$ be the open unit disk (centered at the origin of the complex
plane), and let ${\TT=\partial\DD}$ be the unit circle, where
$\DD^-=\DD\cup\TT$ is the closed unit disk.

%%%%%%%%%%%%%%%%%%%%%%%%%%%%%%%%%%%%%%%%%%%%%%%%%%%%%%%%%%%%%% SECTION 3
\section{Preliminaries}

Consider the above set-up$.$ Take an arbitrary operator ${T\in\BH}$.

%%%%%%%%%%%%%%%%%%%%%%%%%%%%%%%%%%%%%%%% PROPOSITION 3.1
\begin{proposition}
If\/ $T$ is a contraction, then\/ ${\sigma_{\kern-1ptR}(T)\sse\DD}$.
\end{proposition}

\begin{proof}
Take a contraction $T$ acting on a Hilbert space$.$ Consider the
Nagy--Foia\c s--Langer decomposition
$$
T=C\oplus U
$$
of $T$ (see e.g., \cite[Theorem 5.1]{MDOT})$.$ Here $C$ is a completely
nonunitary contraction, $U$ is unitary, and $\oplus$ stands for
orthogonal direct sum$.$ Since $C$ is a completely nonunitary contraction,
$$
\sigma_{\kern-1ptP}(C)\cup\sigma_{\kern-1ptR}(C)\sse\DD
$$
\cite[Corollary 7.4 and Proposition 8.4]{MDOT}$.$ Recall that the point
spectrum of an orthogonal direct sum of operators is the union of the
point spectra$.$ Thus
$$
\sigma_{\kern-1ptP}(C\oplus U)
=\sigma_{\kern-1ptP}(C)\cup\sigma_{\kern-1ptP}(U)
\quad\;[\,\hbox{and}\;\quad
\sigma_{\kern-1ptP}((C\oplus U)^*)
=\sigma_{\kern-1ptP}(C^*)\cup\sigma_{\kern-1ptP}(U^*)\,].
$$
Since $U$ is normal,
$\sigma_{\kern-1ptR}(U)=\sigma_{\kern-1ptR}(U^*)=\void$, and so
$$
\sigma_{\kern-1ptP}(U^*)^*=\sigma_{\kern-1ptP}(U).
$$
\vskip-4pt\noi
Therefore,
$$
\sigma_{\kern-1ptR}(T)
=[\sigma_{\kern-1ptP}(C^*)^*\cup\sigma_{\kern-1ptP}(U)]
\\[\sigma_{\kern-1ptP}(C)\cup\sigma_{\kern-1ptP}(U)]
\sse\sigma_{\kern-1ptR}(C)\sse\DD.
$$
\vskip-17pt\noi
\end{proof}

\vskip4pt
For another proof, still more elementary than the above one see, for
instance, \cite[Proposition 8.5]{MDOT}$.$ Similarity to a contraction
implies power boundedness and similarity preserves the spectrum and its
parts$.$ In particular, if $\widetilde T$ is similar to $T$, then
$\sigma_{\kern-1ptR}(\widetilde T)=\sigma_{\kern-1ptR}(T).$ Therefore
Proposition 3.1 says
$$
T\;\hbox{is similar to a contraction}
\quad\limply\quad
\sigma_{\kern-1ptR}(T)\sse\DD.
$$
Thus it has been asked in \cite[p.114]{MDOT} whether Proposition 3.1 (or
its equivalent form in the above displayed implication) can be extended to
power bounded operators$:$ {\it does power boundedness imply inclusion of
the residual spectrum in the open unit disk}$\,?$ In other words, if
${\sup_n\|T^n\|<\infty}$, is it true that
${\sigma_{\kern-1ptR}(T)\sse\DD}\/?$ The purpose of the note is to
answer this question.

%%%%%%%%%%%%%%%%%%%%%%%%%%%%%%%%%%%%%%%%%%%%%%%%%%%%%%%%%%%%%% SECTION 4
\section{Answer}

Consider the above set-up$.$ Take an arbitrary operator ${T\in\BH}$.

%%%%%%%%%%%%%%%%%%%%%%%%%%%%%%%%%%%%%%%% THEOREM 4.1
\begin{theorem}
If\/ $T$ is power bounded, then\/ ${\sigma_{\kern-1ptR}(T)\sse\DD}$.
\end{theorem}

\begin{proof}
Consider the Ces\`aro means associated with the operator ${T\in\BH}$,
$$
C_n=\smallfrac{1}{n}\sum_{k=0}^{n-1}T^k\,\in\,\BH.
$$
Recall$:$ {\it If\/ $T$ is power bounded, then the sequence of Ces\`aro
means\/ $\{C_n\}$ converges strongly}$.$ [In other words, every power
bounded operator is (strongly) ergodic].

\vskip6pt\noi
This is the well-known Mean Ergodic Theorem for power bounded operators
(which holds on reflexive Banach spaces; see e.g.,
\cite[Corollary VIII.5.4]{LO1})$.$ The $\BH$-valued sequence $\{C_n\}$
converges strongly, which means the $\H$-valued sequence $\{C_nx\}$
converges in $\H$ for every ${x\in\H}.$ The Banach--Steinhaus Theorem
ensures the existence of an operator ${E\in\BH}$ for which ${C_nx\to Ex}$
for every ${x\in\H}.$ Notation:
$$
C_n \sconv E.
$$
Since each $C_n$ is a polynomial in $T$, $C_n$ and $T$ commute, and so
$E$ and $T$ commute$.$ Since $C_nT=C_n+\frac{1}{n}T^n-\frac{1}{n}I$ and
$T$ is power bounded, $ET=E.$ Therefore,
$$
ET=TE=E.
$$
If $T$ is power bounded, then ${r(T)\le1}$ where $r(T)$ stands for the
spectral radius of $T$ (see e.g., \cite[p.10]{MDOT}), which implies
${\sigma_{\kern-1ptR}(T)\sse\DD^-}\!.$ Suppose
$$
\sigma_{\kern-1ptR}(T)\cap\TT\ne\void.
$$
Take any ${\lambda\in\sigma_{\kern-1ptR}(T)\cap\TT}.$ Thus ${|\lambda|=1}$,
$$
T^*x_0=\overline\lambda x_0
\;\;\hbox{for some nonzero vector ${x_0\in\H}$}
\qquad ({\rm i.e.,}\;\lambda\in\sigma_{\kern-1ptP}(T^*)^*),
$$
$$
Tx\ne\lambda x
\;\;\hbox{for every nonzero vector ${x\in\H}$}
\qquad ({\rm i.e.,}\;\lambda\not\in\sigma_{\kern-1ptP}(T)).
$$
Set
$$
T^\cd=\smallfrac{1}{{\;\overline\lambda^{\phantom|}}}T^*\,\in\,\BH,
$$
which is power bounded since $T$ is and ${|\lambda|=1}.$ So there is
an ${E^\cd\!\in\BH}$ such that
$$
C^\cd_n=\smallfrac{1}{n}\sum_{k=0}^{n-1}T^{\cd k}\sconv E^\cd\!,
\qquad
E^\cd T^\cd=T^\cd E^\cd=E^\cd
\quad\;\hbox{and}\;\quad
E^\cd x_0=C^\cd_nx_0=x_0.
$$
(Indeed, since $T^{*k}x_0=\overline\lambda^k x_0$ we get
$C^\cd_nx_0=\frac{1}{n}\sum_{k=0}^{n-1}T^{\cd k}x_0=x_0.$) Moreover,
set
$$
y_0=E^{\cd*}x_0\,\in\,\H,
$$
which is nonzero since $x_0$ is$:$
$0\ne\|x_0\|=\<x_0\,;x_0\>=\<E^\cd x_0\,;x_0\>=\<x_0\,;y_0\>.$ Hence
$$
\<Ty_0\,;x\>
\kern-1pt=\kern-1pt\<TE^{\cd*}x_0\,;x\>
\kern-1pt=\kern-1pt \<x_0\,;E^\cd T^*x\>
\kern-1pt=\kern-1pt\<x_0\,;\overline\lambda E^\cd T^\cd x\>
\kern-1pt=\kern-1pt\<x_0\,;\overline\lambda E^\cd x\>
\kern-1pt=\kern-1pt\<\lambda y_0\,;x\>
$$
for every ${x\in\H}$, and so
$$
Ty_0=\lambda y_0.
$$
Since ${y_0\ne0}$ this contradicts the fact that ${Tx\ne\lambda x}$ for
every ${0\ne x\in\H}.$ Then the assumption
$\sigma_{\kern-1ptR}(T)\cap\TT\ne\void$ fails, and consequently
${\sigma_{\kern-1ptR}(T)\sse\DD}$.
\end{proof}

%%%%%%%%%%%%%%%%%%%%%%%%%%%%%%%%%%%%%%%%%%%%%%%%%%%%%%%%%%%%%% SECTION 5
\section{Remarks}

Let $\sigma(T)$ be the spectrum of an operator $T$ in $\BH$ and consider
its continuous spectrum
$\sigma_{\kern-.5ptC}(T)
=\sigma(T)\\[{\sigma_{\kern-1ptP}(T)\cup\sigma_{\kern-1ptR}(T)}]$
so that
$\{\sigma_{\kern-1ptP}(T),\sigma_{\kern-1ptR}(T),\sigma_{\kern-.5ptC}(T)\}$
is a classical partition of $\sigma(T)$.

\vskip6pt
Recall$:$ (i) similarity preserves the spectrum and its parts (i.e., if
${W\!\in\BH}$ is invertible with an inverse $W^{-1}$ in $\BH$, then the
parts of the spectrum of $WTW^{-1}$ coincide with the respective parts of
the spectrum of $T$), and (ii) similar to a power bounded is again power
bounded$.$ Hence Theorem 4.1 says
$$
T\;\hbox{is similar to a power bounded}
\quad\limply\quad
\sigma_{\kern-1ptR}(T)\sse\DD.                              \leqno{\rm(a)}
$$
Thus everything that has been said above (and below) about power bounded
operators applies ``ipsi literis'' to operators similar to a power
bounded operator.

\vskip6pt
Theorem 4.1 also characterizes the peripheral spectrum (i.e.,
${\sigma(T)\cap\TT}$) of a power bounded operator:
$$
T\;\hbox{is power bounded}
\quad\limply\quad
\sigma(T)\cap\TT\sse\sigma_{\kern-1ptP}(T)\cup\sigma_C(T).  \leqno{\rm(b)}
$$
However, the boundary $\partial\sigma(T)$ of the spectrum $\sigma(T)$ of
a power bounded operator is not necessarily included in
$\sigma_{\kern-1ptP}(T)\cup\sigma_C(T).$ For instance take a unilateral
weighted shift $T={\rm shift}\{w_k\}$ in $\B[\ell_+^2]$ with weighting
sequence $\{w_k\}$ such that $w_k=\frac{1}{k}$ for each positive integer
$k.$ This is a quasinilpotent compact contraction (thus power bounded)
with $\|T\|=1$ for which
$$
\partial\sigma(T)=\sigma(T)=\sigma_{\kern-1ptR}(T)=\0.
$$

\vskip4pt
On the other hand Theorem 4.1 leads to a conjugate symmetry between
the intersection of the unit circle with the point spectra of $T$ and
$T^*\!$ (so that eigen\-values in the unit circle of a power bounded
operator are normal eigenvalues), namely,
$$
T\;\hbox{is power bounded}
\quad\limply\quad
\sigma_{\kern-1ptP}(T^*)^*\cap\TT
=\sigma_{\kern-1ptP}(T)\cap\TT.                             \leqno{\rm(c)}
$$
Indeed, set $A=\sigma_{\kern-1ptP}(T)$ and $B=\sigma_{\kern-1ptP}(T^*)$
so that $\sigma_{\kern-1ptR}(T)={B^*\\A}$ and
$\sigma_{\kern-1ptR}(T^*)={A^*\\B}.$ Suppose $T$ is power bounded$.$ Thus
$T^*$ is again power bounded$.$ By Theorem 4.1,
${(B^*\\A)\cap\TT}={(A^*\\B)\cap\TT}=\void.$ Equivalently,
${(B^*\\A)\cap\TT}=\void$ and ${(A^*\\B)^*\cap\TT}=\void$,
(i.e., ${(A\\B^*)\cap\TT}=\void).$ Therefore
${(B^*\cap\TT)\\(A\cap\TT)}=\void$ and ${(A\cap\TT)\\(B^*\cap\TT)}=\void$,
which means ${B^*\cap\TT}={A\cap\TT}$.

%%%%%%%%%%%%%%%%%%%%%%%%%%%%%%%%%%%%%%%%%%%%%%%%%%%%%%%%%%%%%% SECTION 6
\section{Applications}

Take an arbitrary ${T\in\BH}.$ Let $\R(\lambda I-T)={(\lambda I-T)(\H)}$
denote the range of ${\lambda I-T}$ for an arbitrary ${\lambda\in\CC}$
and consider the set
$$
\M(T)
=\big\{u\in\H\!:\,\sup_n\big\|\hbox{$\sum$}_{k=0}^nT^ku\big\|<\infty\big\},
$$
which are linear manifolds of $\H$ invariant under $T.$ Take
${y\in\R(I-T)}$ arbitrary so that ${y=Tx-x}$ for some ${x\in\H}.$ Then
for each nonnegative integer $n$
$$
\hbox{$\sum$}_{k=0}^nT^ky=T^{n+1}x-x.                           \eqno(*)
$$

%%%%%%%%%%%%%%%%%%%%%%%%%%%%%%%%%%%%%%%% PROPOSITION 6.1
\begin{proposition}
$\R(I-T)\sse\M(T)$ if and only if\/ $T$ is power bounded.
\end{proposition}

\begin{proof}
Take any ${x\in\H}$ and set ${y=Tx-x}$ so that ${y\in\R(I-T)}.$ If
$\R({I-T})\sse\M(T)$, then $\sup_n\|{T^{n+1}x-x}\|<\infty$ according to
$(*)$, and hence ${\sup_n\|T^nx\|<\infty}.$ Since this holds for every
${x\in\H}$, $T$ is power bounded by the Banach--Steinhaus Theorem$.$
Conversely, if $T$ is power bounded, then $\R({I-T})\sse\M(T)$ by $(*)$.
\end{proof}

%%%%%%%%%%%%%%%%%%%%%%%%%%%%%%%%%%%%%%%% COROLLARY 6.1
\begin{corollary}
If\/ $\R(I-T)\sse\M(T)$, then\/ $\sigma_{\kern-1ptR}(T)\sse\DD$.
\end{corollary}

\begin{proof}
This follows by Proposition 6.1 and Theorem 4.1.
\end{proof}

By Proposition 6.1 $T$ is power unbounded if and only if
${\R(I-T)\not\sse\M(T)}.$
\vskip6pt\noi
{\narrower \narrower
{\it In particular, if\/ a power unbounded operator is such that
$\sup_n\|T^nx\|=\infty$ for every nonzero\/ ${x\in\H}$, then}\/
${\M(T)\cap\R(I-T)}=\0$\/.
\vskip6pt}\noi
Indeed, if ${\M(T)\cap\R(I-T)}\ne\0$ and ${0\ne y\in\M(T)\cap\R(I-T)}$,
then by $(*)$ there exists ${0\ne x\in\H}$ such that
${\|T^{n+1}x-x\|}=\|\sum_{k=0}^nT^ky\|$ for each nonnegative integer $n$,
and hence ${\sup_n\|T^nx\|<\infty}$ because ${y\in\M(T)}$.
\vskip6pt\noi
{\narrower\narrower
{\it More particularly, there are power unbounded operators with\/
$\sup_n\|T^nx\|$ $=\infty$ for every nonzero\/ ${x\in\H}$ for which}\/
${\0=\M(T)\subset\R(I-T)=\H}$.
\vskip6pt}\noi
(Set ${T=2I}$.) On the other hand,
\vskip6pt\noi
{\narrower\narrower
{\it there exist power bounded operators for which}\/ ${\R(I-T)=\M(T)}$.
\vskip6pt}\noi
(Trivial examples$:$ ${T=O,I,\frac{1}{2}I}$ --- less trivial$:$
$T={(I\oplus\frac{1}{2}I)},{(O\oplus\frac{1}{2}I})$.)

\vskip6pt
The noninclusion $\M(T)\not\sse\R(I-T)$ can be characterized as follows.

%%%%%%%%%%%%%%%%%%%%%%%%%%%%%%%%%%%%%%%% PROPOSITION 6.2
\begin{proposition}
${\M(T)\not\sse\R(I-T)}$ if and only if there exists a nonzero operator\/
${L\in\BH}$ for which\/ ${\R(L)\cap\R(I-T)}=\0$ and\/ ${\R(L)\sse\M(T)}$.
\end{proposition}

\begin{proof}
If ${T\in\BH}$ is such that ${\M(T)\not\sse\R(I-T)}$, then take an
arbitrary nonzero ${y\in\M(T)\\\R(I-T)}$, and consider the orthonormal
projection ${E\in\BH}$ for which $\R(E)=\span\{y\}$ so that
${\R(E)\cap\R(I-T)}=\0$ and ${\R(E)\sse\M(T)}$ (since $\R(E)$, $\R({I-T})$,
and $\M(T)$ are linear manifolds of $\H).$ Conversely, if $L$ and $T$ in
$\BH$ are such that ${\0\ne\R(L)\sse\M(T)\sse\R(I-T)}$, then
${\R(L)\cap\R(I-T)}\ne\0$.
\end{proof}

And the inclusion $\R(L)\sse\M(T)$ is characterized as follows.

%%%%%%%%%%%%%%%%%%%%%%%%%%%%%%%%%%%%%%%% PROPOSITION 6.3
\begin{proposition}
$\R(L)\sse\M(T)$ if and only if\/
$\sup_n\big\|\sum_{k=0}^nT^kL\big\|<\infty$.
\end{proposition}

\begin{proof}
Take arbitrary operators $L$ and $T$ in $\BH.$ By the Banach--Steinhaus
Theorem, $\sup_n\big\|\sum_{k=0}^nT^kL\big\|<\infty$ if and only if
$\sup_n\big\|\sum_{k=0}^nT^kLx\big\|<\infty$ for every ${x\in\H}.$
Equivalently, $\sup_n\big\|\sum_{k=0}^nT^ky\big\|<\infty$ for every
${y\in\R(L)}$, which means ${\R(L)\sse\M(T)}$.
\end{proof}

As it is well known, the spectral radius
$r(T)=\max_{\lambda\in\sigma(T)}|\lambda|$ of an operator $T$ on a Banach
space is less than 1 if and only if the operator $T$ is uniformly stable
(notation$:$ ${T^n \uconv O}$), which completely characterizes uniform
stability; that is,
$$
\|T^n\|\to0
\quad\hbox{if and only if}\quad
r(T)<1.
$$
Clearly, uniform stability implies power boundedness, which implies
${r(T)\le1}$$.$ The above equivalence has been extended in \cite{Lek} where
it was given a complete characterization (in terms of an ergodic condition
and the peripheral spectrum) for uniform convergence to zero of $\{T^nL\}$
for operators $L$ in the commutant of a power bounded operator $T.$ We
state the result from \cite{Lek} below$.$ Take any operator ${T\in\BH}$ and
let $\{T\}'$ denote its commutant (the subalgebra of $\BH$ consisting of
all operators ${L\in\BH}$ that commute with $T$).

%%%%%%%%%%%%%%%%%%%%%%%%%%%%%%%%%%%%%%%% PROPOSITION 6.4
\begin{proposition}
\cite{Lek} If\/ ${T\in\BH}$ is a power bounded operator and\/ $L$ lies
in\/ $\{T\}'$, then\/ ${\|T^nL\|\to0}$ if and only if\/
$\frac{1}{n+1}
\big\|\sum_{k=0}^n\big(\frac{T}{\lambda}\big)^k L\big\|\to0$
for every\/ ${\lambda\in\sigma(T)\cap\TT}$.
\end{proposition}

The above statement is vacuous for the particular case of $L$ being an
isometry in $\{T\}'$ (e.g., for ${L=I}).$ Indeed, if $L$ is an isometry
then $\|T^nL\|=\|T^n\|$, and ${r(T)<1}$ makes the peripheral spectrum
${\sigma(T)\cap\TT}$ empty$.$ In fact, ${r(T)<1}$ if and only if
${\|T^n\|\to0}$, which implies ${\|T^nL\|\to0}$ for all ${L\in\BH}.$ Thus
we assume the power bounded operator $T$ is not uniformly stable;
equivalently, the power bounded operator $T$ is such that $r(T)=1.$ Also,
if $T$ is an isometry and ${\|T^nL\|\to0}$, then $L=O.$ Actually, if $L$
is not injective (i.e., if $\N(L)\ne\0$), then there exists ${0\ne x_0}$
in $\H$ for which
$\|T^nLx_0\|=\big\|\sum_{k=0}^n(\frac{T}{\lambda})^kLx_0\big\|=0$ for
every ${n\ge0}$ and every ${\lambda\in\CC}$.

%%%%%%%%%%%%%%%%%%%%%%%%%%%%%%%%%%%%%%%% COROLLARY 6.2
\begin{corollary}
Let\/ ${T\in\BH}$ be a power bounded operator with\/ ${r(T)=1}$ and let\/
$L$ be an injective operator in\/ $\{T\}'.$ Then
$$
\|T^nL\|\to0
\,\;\hbox{if and only if}\;\;
\smallfrac{1}{n+1}
\big\|{\sum}_{k=0}^n\big(\smallfrac{T}{\lambda}\big)^kL\big\|\to0
\,\;\hbox{for every}\;\;
\lambda\in\sigma_{\kern-1pt C}(T)\cap\TT,
$$
and\/ $\sigma(T)\cap\TT=\sigma_{\kern-1pt C}(T)\cap\TT\ne\void$ whenever
any of the above limits hold.
\end{corollary}

\begin{proof}
If ${r(T)\kern-.5pt=\kern-.5pt1}$, then
${\sigma(T)\cap\TT\kern-1pt\ne\kern-1pt\void}.$ If $T$ is a power bounded,
then ${\sigma_{\kern-1pt R}(T)\cap\TT\kern-.5pt=\kern-.5pt\void}$
according to Theorem 4.1$.$ If ${L\in\{T\}'}$ is injective and the claimed
equivalence holds, then ${\sigma_{\kern-1pt P}(T)\cap\TT=\void}.$ Indeed,
if ${L\in\{T\}'}$ and ${\sigma_{\kern-1pt P}(T)\cap\TT\ne\void}$, then
there is an ${\lambda\in\TT}$ such that ${Tx_0=\lambda x_0}$ for some
${0\ne x_0\in\H}$ (thus $x_0={T^kx_0/\lambda^k}$). Hence
$\|Lx_0\|={\|T^nLx_0\|\to0}$ (also ${\|Lx_0\|}
=\frac{1}{n+1}\big\|\sum_{k=0}^nLx_0\big\|
=\frac{1}{n+1}\big\|\sum_{k=0}^n\big(\frac{T}{\lambda}\big)^kLx_0\big\|
\to0)$
and so ${0\ne x_0\in\N(L)}$, which implies $L$ is not injective$.$ Thus
$$
\sigma(T)\cap\TT=\sigma_{\kern-1pt C}(T)\cap\TT\ne\void,
$$
and the claimed result follows from Proposition 6.4.
\end{proof}

An important particular case of Proposition 6.4 is that of ${L=I-T}$ so
that $T^nL=T^n-T^{n+1}$ for every integer ${n\ge0}.$ This leads to a
classical result from \cite{KT} (which holds for contractions acting on
Banach spaces).

%%%%%%%%%%%%%%%%%%%%%%%%%%%%%%%%%%%%%%%% PROPOSITION 6.5
\begin{proposition}
\cite{KT} Let\/ ${T\in\BH}$ be a contraction$.$ Then\/
${\|T^n-T^{n+1}\|\to0}$ if and only if\/ ${\sigma(T)\cap\TT}\sse\{1\}$.
\end{proposition}

If the contraction $T$ is uniformly stable (i.e., if ${r(T)<1}$), then
${\sigma(T)\cap\TT=\void}.$ Otherwise it is a normaloid contraction of
unit norm, which means $r(T)=\|T\|=1$.

%%%%%%%%%%%%%%%%%%%%%%%%%%%%%%%%%%%%%%%% COROLLARY 6.3
\begin{corollary}
If ${r(T)=1}$, then the above unique point ${\lambda=1}$ in the peripheral
spectrum ${\sigma(T)\cap\TT}$ is either an eigenvalue {\rm(i.e.,} lies in
$\sigma_{\kern-1pt P}(T))$, or lies in $\sigma_{\kern-1pt C}(T)$.
\end{corollary}

\begin{proof}
If ${\N(I-T)\ne\0}$ (i.e., if ${I-T}$ is not injective), then
${1\in\sigma_{\kern-1pt P}(T)}.$ Otherwise ${1\in\sigma_{\kern-1pt C}(T)}$
because ${\sigma_{\kern-1pt R}(T)\cap\TT=\void}$ by Proposition 3.1 since
$T$ is a contraction (or by Theorem 4.1 since contractions are power
bounded).
\end{proof}

%%%%%%%%%%%%%%%%%%%%%%%%%%%%%%%%%%%%%%%%%%%%%%%% ACKNOWLEDGMENT
\vskip-5pt\noi
\section*{Acknowledgment}
We thank an anonymous referee who made sensible observations throughout
the text contributing to the improvement of the paper$.$ We also thank an
anonymous person who posted the idea behind the argument in
the proof of Theorem 4.1 at https://mathoverflow.net/questions/204438/

%%%%%%%%%%%%%%%%%%%%%%%%%%%%%%%%%%%%%%%%%%%%%%%% REFERENCES
\vskip-10pt\noi
\bibliographystyle{amsplain}

\end{document}